\documentclass[a4paper,10pt]{amsart}
\usepackage[utf8]{inputenc}

\usepackage{amsmath, amssymb, amsthm}
\newtheorem{prop}{Proposition}
\newtheorem*{thm}{Theorem}
\newtheorem*{cor}{Corollary}
\theoremstyle{remark}
\newtheorem*{ex}{Example}

\title[An algorithm for the topological Euler characteristic]{An algorithm for computing \\ the topological Euler characteristic \\ of complex projective varieties}
\author{Christine Jost}

    \address{Stockholm University, Department of
             Mathematics, SE-106 91 Stockholm, Sweden}
    \email{jost@math.su.se}
    \urladdr{http://www.math.su.se/$\sim$jost}

\thanks{}

\begin{document}

\begin{abstract}
We present an algorithm for the symbolic and numerical computation of the degrees of the Chern-Schwartz-MacPherson classes of a closed subvariety of projective space $\mathbb{P}^n_\mathbb{C}$. As the degree of the top Chern-Schwartz-MacPherson class is the topological Euler characteristic, this also yields a method to compute the topological Euler characteristic of projective varieties. The method is based on Aluffi's symbolic algorithm to compute degrees of Chern-Schwartz-MacPherson classes, a symbolic method to compute degrees of Segre classes, and the regenerative cascade by Hauenstein, Sommese and Wampler.  The new algorithm complements the existing algorithms. We also give an example for using a theorem by Huh to compute an invariant from algebraic statistics, the maximum likelihood degree of an implicit model.
\end{abstract}

\maketitle

\section{Introduction}

The topological Euler characteristic is a widely studied invariant in many areas of mathematics. In this article, we present an algorithm for the symbolic and numerical computation of the topological Euler characteristic of a possibly singular closed subvariety of $\mathbb{P}^n_\mathbb{C}$. More generally, we present a method to compute the degrees of the Chern-Schwartz-MacPherson classes of such a variety. As the degree of the top Chern-Schwartz-MacPherson class is the topological Euler characteristic, this yields a method to compute the topological Euler characteristic for projective varieties. 

The method we describe is based on several different known methods. According to Aluffi's article \cite{A}, degrees of the Chern-Schwartz-MacPherson classes can be computed from so-called shadows of graphs, also called projective degrees of a map. Using other results from \cite{A}, we prove in this article that the shadows of graphs can be computed from the degrees of Segre classes, which in turn can be computed from so-called residuals, according to \cite{EJP}. The residuals can be computed symbolically, but are also a by-product of the regenerative cascade, a method for computing the numerical irreducible decomposition of a variety developed in \cite{HSW}.

The existing methods to compute the topological Euler characteristic of varieties are, to the author's knowledge, all symbolic. The algorithm developed in \cite{A} was implemented by Aluffi in a Macaulay2 package called CSM. Also, Macaulay2 \cite{M2} provides the command euler which computes the Euler characteristic using the computation of Hodge numbers. However, the latter only works for nonsingular varieties. Moreover, in \cite{MB} an algorithm for the computation of degrees of Chern-Schwartz-MacPherson classes of algebraic subsets of $\mathbb{C}^n$ is presented. 

The algorithm described in this article is implemented in version 0.2 and higher of the Macaulay2 package CharacteristicClasses. More details about the implementation can be found in \cite{J}.

We give a short outline of the content of this article. In Section 2, the definitions of Chern-Schwartz-MacPherson classes, Segre classes and their degrees are recalled. In Section 3, we describe how to compute degrees of Chern-Schwartz-MacPherson classes from the degrees of Segre classes, and in Section 4 how to compute degrees of Segre classes numerically using the regenerative cascade. In Section 5, the running times of our implementation are compared to existing implementations. We conclude with an example from algebraic statistics in Section 6. 


\section{Segre classes, Chern-Schwartz-MacPherson classes and their degrees}

We work over the field of complex numbers. In this and the following section we will use the language of schemes instead of varieties. Let $X$ be a possibly singular closed subscheme of $\mathbb{P}^n$, embedded by $i\colon X \hookrightarrow \mathbb{P}^n$, and let $k$ be the dimension of $X$. We will recall the definitions of the Chern-Schwartz-MacPherson classes and Segre classes of $X$ as well as their degrees. The standard reference for intersection theory, including Chow groups and Segre classes is \cite{F}. Chern-Schwartz-MacPherson classes are described in several articles by Aluffi, e.g.\ \cite{A}.

Characteristic classes, such as Chern-Schwartz-MacPherson and Segre classes, of $X$ are elements of the Chow group $A_*(X)=\bigoplus_{j=0}^kA_j(X)$ of $X$, i.e., they are cycles modulo rational equivalence. However, the Chow group may be hard to compute in general. Hence computational methods for characteristic classes focus on computing the \emph{degrees} of the respective classes rather than the classes themselves. Let $\alpha = \sum_{i=0}^r \alpha_i [V_i]$ be a $d$-dimensional cycle in $A_{d}(X)$, then its degree is defined as $\mathrm{deg}( \alpha ) = \sum_{i=0}^r \alpha_i \, \mathrm{deg}(V_i)$, where $\mathrm{deg}(V_i)$ is the degree of $V_i$ considered as a subvariety of $\mathbb{P}^n$. Equivalently, one can consider the pushforward $i_*(\alpha)$ of $\alpha$ to the Chow group $A_*(\mathbb{P}^n)$ of $\mathbb{P}^n$. Let $H$ be the hyperplane class, then $A_*(\mathbb{P}^n) = \mathbb{Z}[H]/(H^{n+1})$ and $i_*(\alpha) = H^{n-d}\sum_{i=0}^r \alpha_i \mathrm{deg}(V_i)$. Observe that the definition of degree 
does 
not agree with the 
usual definition of the degree of cycle classes denoted by $\int$. 

The total Chern class of a nonsingular variety is defined to be the total Chern class of its tangent bundle. There are several generalizations of this concept to possibly singular schemes, including Chern-Schwartz-MacPherson classes. They were defined independently by MacPherson \cite{M}, proving a conjecture of Grothendieck, and Schwartz \cite{S}, and were shown to agree in \cite{BS}. Chern-Schwartz-MacPherson classes enjoy nice functorial properties, which we recall  by summarizing sections 2.2 and 2.3 of \cite{A}. Let $S$ be a proper scheme. The Chern-Schwartz-MacPherson class of a closed subscheme $X$ of $S$ is an element $c_\mathrm{SM}(X)$ in $A_*(X)$ such that $c_\mathrm{SM}(X) = c(X) \cap [X]$, the total Chern class of $X$, for nonsingular schemes $X$. The $i$-th Chern-Schwartz-MacPherson class $(c_\mathrm{SM})_i(X) \in A_{k-i}(X)$ is then the codimension $i$ part of the total Chern-Schwartz-MacPherson class $c_\mathrm{SM}(X)$. 
Chern-Schwartz-MacPherson classes extend to constructible functions by $c_\mathrm{SM}(\sum_{V \subset S} m_V \mathbf{1}_V) = \sum_{V \subset S} m_V c_\mathrm{SM}(V)$. This actually gives a natural transformation $\mathcal{C} \rightsquigarrow \mathcal{A}$ from the functor of 
constructible functions $\mathcal{C}$ to the Chow group functor $\mathcal{A}$. 
 The functor $\mathcal{C}$ of constructible functions maps a scheme to the abelian group of constructible functions on it. A morphism $f:S \rightarrow T$ of schemes is mapped to the morphism $\mathcal{C}(f)$ of abelian groups by $\mathcal{C}(f)(\mathbf{1}_V)(y) = \chi(f^{-1}(y)\cap V)$, for $V \subseteq S$ a subscheme and $y \in T$ a closed point. 
 Here $\chi$ is the topological Euler characteristic and hence $\chi(f^{-1}(y)\cap V)$ the topological Euler characteristic of the fiber of the point $y$.
As a special case of the functorial properties, we get that Chern-Schwartz-MacPherson classes compute the Euler characteristic of the support of schemes. Let $\kappa\colon X \rightarrow \mathrm{point}$. Then
\begin{equation*}
 \kappa_*c_\mathrm{SM}(X) = c_\mathrm{SM}(\mathcal{C}(\kappa) \mathbf{1}_X) = c_\mathrm{SM}(\chi(X_\mathrm{red})\mathbf{1}_\mathrm{point}) 
= \chi(X_\mathrm{red})[\mathrm{point}],
\end{equation*}
hence $\int c_\mathrm{SM}(X) = \chi(X_\mathrm{red})$, where $\int$ as usual denotes the degree of the top class. Observe that constructible functions follow laws of exclusion-inclusion, e.g.,
\begin{equation*}
 \mathbf{1}_{X_1 \cap X_2} = \mathbf{1}_{X_1} + \mathbf{1}_{X_2}  - \mathbf{1}_{X_1 \cup X_2}.
\end{equation*}
It follows that Chern-Schwartz-MacPherson classes follow similar laws of exclusion-inclusion, e.g.,
\begin{equation*}
 c_\mathrm{SM}(X_1 \cap X_2) = c_\mathrm{SM}(X_1) + c_\mathrm{SM}(X_2) - c_\mathrm{SM}(X_1 \cup X_2).
\end{equation*}

Finally, we recall the definition of Segre classes of closed subschemes of $\mathbb{P}^n$.
The $i$-th Segre class $s_i(X, \mathbb{P}^n)$ of $X$ embedded in $\mathbb{P}^n$ is defined to be the $i$-th Segre class of the normal cone $C_X\mathbb{P}^n$ of $X$ in $\mathbb{P}^n$. Instead of using the definition of Segre classes of cones, we use a shortcut, allowing us to define the Segre classes $s_i(X, \mathbb{P}^n) \in A_{k-i}$ for $0 \le i \le k$ directly. Let $E_X$ be the exceptional divisor of the blow-up $\mathrm{Bl}_X\mathbb{P}^n$ of $\mathbb{P}^n$ along $X$, and let $\eta\colon E_X \rightarrow X$ be the projection. Then $ s_i(X,\mathbb{P}^n) := (-1)^p \eta_*(E_X^p) $ for $p=n+i-k$, where $E_X^p$ denotes the $p$-th self intersection of the exceptional divisor $E_X$. The total Segre class $s(X,\mathbb{P}^n)$ of $X$ in $\mathbb{P}^n$ is then defined as $s(X,\mathbb{P}^n) = 1 +  s_1(X,\mathbb{P}^n) + \ldots +  s_k(X,\mathbb{P}^n)$.

\section{Computing degrees of Chern-Schwartz-MacPherson classes using degrees of Segre classes}

In this section we combine results from \cite{A} with a theorem proved here to give an algorithm for the computation of the degrees of Chern-Schwartz-MacPherson classes, provided an algorithm for the computation of degrees of Segre classes. As in the previous section, we consider a $k$-dimensional closed subscheme $X$ of $\mathbb{P}^n$, embedded by $i\colon X \hookrightarrow \mathbb{P}^n$. Recall that computing the degrees of the Chern-Schwartz-MacPherson classes is equivalent to computing the pushforward $i_*c_\mathrm{SM}(X)$  of the total Chern-Schwartz-MacPherson class to the Chow group $A_*(\mathbb{P}^n)$ of $\mathbb{P}^n$.

Observe first that it suffices to find an algorithm computing the Chern-Schwartz-MacPherson classes of hypersurfaces. The Chern-Schwartz-MacPherson
classes of lower-dimensional schemes can then be computed using the exclusion-inclusion principle described in the previous section:
\begin{equation*}
 i_*c_\mathrm{SM}(X_1 \cap X_2) = i_*c_\mathrm{SM}(X_1) + i_*c_\mathrm{SM}(X_2) - i_*c_\mathrm{SM}(X_1 \cup X_2).
\end{equation*}
 Due to a result from \cite{A}, the Chern-Schwartz-MacPherson classes of a hypersurface can be computed from the so-called shadow of the graph of the singular locus of the hypersurface. Reversing another result from \cite{A}, we prove here that this shadow can be computed from the degrees of the Segre classes of the singular locus.
We shortly review the definitions and results from \cite{A}.

Let $X$ be a hypersurface in $\mathbb{P}^n$, given by a homogeneous polynomial $f$ in the polynomial ring $\mathbb{C}[x_0, \ldots, x_n]$. The singular locus of $X$ is given by the ideal $(\frac{\partial f}{\partial x_0}, \ldots, \frac{\partial f}{\partial x_n})$. The blow-up of $\mathbb{P}^n$ along the singular locus of $X$ is a closed subscheme of $\mathbb{P}^n \times \mathbb{P}^n$ which also can be seen as the closure of the graph of $(\frac{\partial f}{\partial x_0}: \ldots: \frac{\partial f}{\partial x_n})\colon \mathbb{P}^n \rightarrow \mathbb{P}^n$. Denote its class in the Chow ring of $\mathbb{P}^n \times \mathbb{P}^n$ by $\Gamma$. The structure theorem for the Chow group of projective bundles (Theorem 3.3 in \cite{F}) assigns to $\Gamma$ a cycle in the Chow group of the first factor $\mathbb{P}^n$. Aluffi calls this cycle the \emph{shadow} of $\Gamma$, and denotes it by $G = \sum_{i=0}^n g_i H^i$. The numbers $g_0, \ldots, g_n$ are also called the projective degrees of the map 
$(\frac{\partial f}{\partial x_0}: \ldots :  \frac{
\partial f}{\partial x_n})$, see chapter 19 of \cite{Ha}. According to Lemma 4.2 in \cite{A2}, the graph of $\Gamma$ can be computed by intersecting $\Gamma$ with the pullback of the 
hyperplane class from the second factor and projecting to the first factor. The importance of the shadow of the graph of the singular locus is that it can be used to compute the degrees of the Chern-Schwartz-MacPherson and Segre classes of $X$, according to the following propositions, Theorem 2.1  and Proposition 3.1 from \cite{A}.
\begin{prop}[Aluffi] \label{CSMfromG} Using the notation above, the pushforward of the Chern-Schwartz-MacPherson class of $X$ is given by
\begin{equation*}
 i_* c_\mathrm{SM} (X) = (1+H)^{n+1} - \sum_{j=0}^n g_j(-H)^j(1+H)^{n-j}.
\end{equation*}
\end{prop}
\begin{prop}[Aluffi] \label{SfromG} Let $Y$ be a subscheme of $\mathbb{P}^n$ given by an ideal with generators all of the same degree $r$, and let $G$ be the shadow of the graph of $Y$. Then the pushforward of the total Segre class of $Y$ is given by
\begin{equation*}
 i_* s(Y,\mathbb{P}^n) = 1 - c(\mathcal{O}(rH))^{-1} \cap (G \otimes \mathcal{O}(rH)).
\end{equation*}
\end{prop}
According to Definition 2 in \cite{A3}, $G \otimes \mathcal{O}(rH) = \sum_{i=0}^n \frac{g_i H^i}{c(\mathcal{O}(rH))^i}$.

We prove a theorem reversing Proposition \ref{SfromG}, i.e., a method to compute the shadow of the graph of the singular locus from the degrees of its Segre classes. Together with Aluffi's results and a method to compute the degrees of the Segre classes of a projective scheme, this yields a method to compute the degrees of the Chern-Schwartz-MacPherson classes of a projective scheme.
\begin{thm}
 Let $X$ be a hypersurface in $\mathbb{P}^n$. Let $\widetilde{s_0}, \ldots, \widetilde{s_n}$ be integers related to the degrees of the Segre classes of the singular locus of $X$ by $\widetilde{s_0} = 1$, $\widetilde{s_1} = \ldots = \widetilde{s_{n-k - 1}} = 0$ and $\widetilde{s_i} = - \mathrm{deg} \ s_{i-(n-k)}(\mathrm{sing}(X),\mathbb{P}^n)$ for $i=n-k, \ldots, n$, where $k=\mathrm{dim}(\mathrm{sing}(X))$. Let the shadow of the graph of $\mathrm{sing}(X)$ be denoted by $G = g_0 + \ldots + g_nH^n$. 
 Then the integers $g_j$ can be computed from the integers $\widetilde{s_i}$ by 
\begin{equation*}
 g_j = \sum_{i=0}^j \binom{j}{i}r^{j-i} \widetilde{s_i},
\end{equation*}
 where $j=0,\ldots,n$ and $r$ is the degree of the generators of the ideal of $\mathrm{sing}(X)$.
\end{thm}
\begin{proof}
 Denote $Y=\mathrm{sing}(X)$. According to Proposition \ref{SfromG}, it holds that
\begin{equation*}
 i_*s(Y, \mathbb{P}) = 1 - c(\mathcal{O}(rH))^{-1} \cap (G\otimes \mathcal{O}(rH)).
\end{equation*}
Recall that $G\otimes \mathcal{O}(rH)$ is defined to be $\sum_{i=0}^n \frac{g_i H^i}{c(\mathcal{O}(rH))^i}$. The identity can be rewritten as
\begin{equation*}
 1 - \sum_{v=0}^k \mathrm{deg}\left(s_v(Y,\mathbb{P}^n)\right) H^{n-k+v} = \frac{1}{1+rH} \sum_{i=0}^n \frac{g_i H^i}{(1+rH)^i} = \sum_{i=0}^n \frac{g_i H^i}{(1+rH)^{i+1}}.
\end{equation*}
We develop the right-hand-side of this identity:
\begin{equation*}
  \sum_{i=0}^n \frac{g_i H^i}{(1+rH)^{i+1}} = \sum_{i=0}^n g_iH^i \left( \sum_{j=0}^n (-rH)^j \right)^{i+1} = \sum_{i=0}^n g_iH^i \sum_{l=0}^n \binom{i+l}{i} (-rH)^l,
\end{equation*}
where the last step uses the following identity for formal sums in a variable $x$:
\begin{equation*}
 \left( \sum_{j=0}^\infty x^j \right)^{i+1} = \sum_{l=0}^\infty \binom{i + l}{i}x^l,
\end{equation*}
for any positive integer $i$.
We continue rewriting the right-hand-side of the equation:
\begin{align*}
 \sum_{i=0}^n g_iH^i \sum_{l=0}^n \binom{i+l}{i} (-rH)^l &=  \sum_{i=0}^n  \sum_{l=0}^n g_i\binom{i+l}{i} (-r)^l H^{i+l}  \\
 &=\sum_{u=0}^n \left( \sum_{t=0}^u \binom{u}{t} g_t (-r)^{u-t} \right) H^u.
\end{align*}
Summarizing, we have proved that
\begin{equation*}
 1 - \sum_{v=0}^k \mathrm{deg}\left(s_v(X,\mathbb{P}^n)\right) H^{n-k+v} = \sum_{u=0}^n \left( \sum_{t=0}^u \binom{u}{t} g_t (-r)^{u-t} \right) H^u.
\end{equation*}
Comparing the coefficients of the powers of $H$ on both sides and changing variables to $i=u=n-k+v$ yields 
\begin{equation*}
 \widetilde{s_i} = \sum_{t=0}^i \binom{i}{t}(-r)^{i-t}g_t.
\end{equation*}
This linear equation system is triangular with ones on the diagonal, hence it is clear that the $g_t$ can be computed from the $\widetilde{s_i}$. Moreover, the matrix of the linear equation system can also be inverted explicitely. It suffices to prove that
\begin{equation}\label{comb}
 \sum_{i=t}^j \binom{j}{i} \binom{i}{t}(-1)^{i-t} = \delta_{jt},
\end{equation}
where $\delta_{jt}$ is the Kronecker delta.
Then it follows that
\begin{align*}
 &\sum_{i=0}^j \binom{j}{i}r^{j-i} \widetilde{s_i} =  \sum_{i=0}^j \binom{j}{i}r^{j-i} \sum_{t=0}^i \binom{i}{t}(-r)^{i-t}g_t \\
&= \sum_{t=0}^j \sum_{i=t}^j \binom{j}{i}  \binom{i}{t}(-1)^{i-t}r^{j-t}g_t
= \sum_{t=0}^j \delta_{jt}r^{j-t}g_t = g_j
\end{align*}
and we are done.

The identity \eqref{comb} is probably well-known. As the author has not been able to find a reference, we prove it for the convenience of the reader. By using the factorial formula for the binomial coefficients, reducing and removing terms not containing $i$ from the summation, one gets
\begin{equation*}
 \sum_{i=t}^j \binom{j}{i} \binom{i}{t}(-1)^{i-t} = \frac{j!}{t!} \sum_{i=t}^j \frac{1}{(j-i)!(i-t)!}(-1)^{i-t}.
\end{equation*}
We expand by $(j-t)!$ and change the summation variable so that $l=i-t$. This yields
\begin{equation*}
 \frac{j!}{t!} \sum_{i=t}^j \frac{1}{(j-i)!(i-t)!}(-1)^{i-t} = \binom{j}{t}\sum_{l=0}^{j-t} \binom{j-t}{l}(-1)^l = \binom{j}{t}\delta_{jt}= \delta_{jt}
\end{equation*}
and we are done.

\end{proof}

\begin{cor}
Let $X$ be a hypersurface in $\mathbb{P}^n$, and let the integers $\widetilde{s_0}, \ldots, \widetilde{s_n}$ be related to the Segre classes of the singular locus of $X$ as in the theorem before. Then the degrees of the Chern-Schwartz-MacPherson classes of $X$ can be computed from the $\widetilde{s_i}$ by
\begin{equation*}
 \mathrm{deg}\ (c_\mathrm{SM})_{p}(X) = \binom{n+1}{q} - \sum_{i=0}^q \widetilde{s_i} \sum_{j=i}^q (-1)^j \binom{j}{i}   \binom{n-j}{q-j}r^{j-i}
\end{equation*}
where $q=n-\mathrm{dim}(X)+p$, $p=0, \ldots, \mathrm{dim}(X)$ and $r$ is the degree of the generators of the ideal of $\mathrm{sing}(X)$.

\end{cor}
\begin{proof}
We use Proposition \ref{CSMfromG} to prove that the Chern-Schwartz-MacPherson classes can be computed from the shadow of the graph $g_i$ by
\begin{equation}\label{csmeq}
 \mathrm{deg}\ (c_\mathrm{SM})_{p}(X) = \binom{n+1}{q} - \sum_{j=0}^q g_j(-1)^j \binom{n-j}{q-j}.
\end{equation} 
Then it follows by the preceeding theorem that
\begin{equation*}
 \mathrm{deg}\ (c_\mathrm{SM})_{p}(X) = \binom{n+1}{q} - \sum_{j=0}^q \sum_{i=0}^j \binom{j}{i}r^{j-i} \widetilde{s_i} (-1)^j \binom{n-j}{q-j}.
\end{equation*}
The claim then follows by changing order of the two summations and rearranging terms.

Proposition \ref{CSMfromG} states that
\begin{equation*}
 i_* c_\mathrm{SM} (X) = (1+H)^{n+1} - \sum_{l=0}^n g_l(-H)^l(1+H)^{n-l}.
\end{equation*}
We rewrite the second term in the right-hand-side of the identity:
\begin{align*}
 &\sum_{l=0}^n g_l(-H)^l(1+H)^{n-l} = \sum_{l=0}^n g_l(-H)^l \sum_{k=0}^{n-l} \binom{n-l}{k} H^k  \\
&= \sum_{l=0}^n \sum_{k=0}^{n-l} g_l(-1)^l \binom{n-l}{k}    H^{k+l} = \sum_{q=0}^n H^q \sum_{j=0}^q g_j(-1)^j \binom{n-j}{q-j}.
\end{align*}
Hence
\begin{equation*}
  i_* c_\mathrm{SM} (X) = \sum_{q=0}^n H^q \left( \binom{n+1}{q} - \sum_{j=0}^q g_j(-1)^j \binom{n-j}{q-j} \right).
\end{equation*}
Comparing the coefficients of the powers of $H$ on both sides yields identity \eqref{csmeq}.
\end{proof}
Recall that it suffices to compute the degrees of the Chern-Schwartz-MacPherson classes for hypersurfaces due to the inclusion-exclusion principle. We hence have found an algorithm to compute the degrees of the Chern-Schwartz-MacPherson classes and hence the topological Euler characteristic of any closed subscheme of $\mathbb{P}^n$, provided an algorithm for the computation of Segre classes.

Observe that the computation of the topological Euler characteristic of projective schemes also allows to compute the Euler characteristic of affine schemes. One just computes the difference between the Euler characteristic of the homogenization and the subscheme at infinity.

\begin{ex}
 We use the preceeding corollary to compute the degrees of the Chern-Schwartz-MacPherson classes of the nodal plane cubic given by the ideal $(x^3 + x^2z - y^2z) \subset \mathbb{C}[x,y,z]$. The singular locus of the curve is the point $P=[0:0:1]$ and the degree of its Segre class $s_0(P,\mathbb{P}^2)$ is the degree of the point, $\mathrm{deg}(s_0(P,\mathbb{P}^2))=1$. Furthermore, the dimension of the ambient space $\mathbb{P}^2$ is $n=2$ and the dimension of the singular locus is $k=0$. Hence $\widetilde{s_0}=1$, $\widetilde{s_1}=0$ and $\widetilde{s_2} = -1$. The generators of the ideal of the singular locus of $(x^3 + x^2z - y^2z)$ have degree 2, so $r=2$.

We start by computing $\mathrm{deg}(c_\mathrm{SM})_0(C)$. Because $p=0$ we have that $q=1$. By the formula in the corollary it then follows that
\begin{align*}
 \mathrm{deg}(c_\mathrm{SM})_0(C) &= \binom{2+1}{1} - \sum_{i=0}^1 \widetilde{s_i} \sum_{j=i}^1 (-1)^j \binom{j}{i} \binom{2-j}{1-j}2^{j-i} \\
 &= 3 - (0 \cdot \widetilde{s_0} - \widetilde{s_1}) = 3.
\end{align*}
In the same way, we compute $\mathrm{deg}(c_\mathrm{SM})_1(C)$, here $p=1$ and $q=2$:
\begin{align*}
 \mathrm{deg}(c_\mathrm{SM})_1(C) &= \binom{2+1}{0} - \sum_{i=0}^2 \widetilde{s_i} \sum_{j=i}^2 (-1)^j \binom{j}{i}   \binom{2-j}{2-j}2^{j-i} \\
 &= 3 - (3 \cdot \widetilde{s_0} +3\cdot \widetilde{s_1} + \widetilde{s_2}) = 1.
\end{align*}
So the degree of the top Chern-Schwartz-MacPherson class of the nodal plane cubic is 1. As the degree of the top Chern-Schwartz-MacPherson class equals the topological Euler characteristic, this confirms that the nodal plane cubic has topological Euler characteristic 1.
The pushforward to the Chow ring of $\mathbb{P}^2$ of the Chern-Schwartz-MacPherson class is $i_*c_\mathrm{SM}(C) = 3H + H^2$, where $H$ is the hyperplane class. 
\end{ex}

\section{Computing degrees of Segre classes using the regenerative cascade}

In this section we describe how to compute degrees of Segre classes numerically. 
The main ingredience of the method is the algorithm in \cite{EJP}, a symbolic algorithm for the computation of the degrees of Segre classes. In more detail, the method in \cite{EJP} reduces computation of degrees of Segre classes to the computation of the degrees of so-called residuals, which can be computed symbolically. In the following, we describe how the regenerative cascade from \cite{HSW} can be used to compute the residuals numerically, which was already claimed in \cite{EJP}.
The regenerative cascade is implemented in Bertini \cite{B}, a software for the numerical solution of polynomial equation systems via homotopy continuation. One of the advantages of homotopy continuation is that the algorithms are parallelizable.

We start by recalling the Segre class algorithm from \cite{EJP}. Let $X$ be a closed $k$-dimensional subvariety of $\mathbb{P}^n$ given by an ideal $I$. Let $m$ be the maximal degree of the generators. Pick $n$ random degree $m$ elements $f_1, \ldots, f_n$ of the ideal $I$, corresponding to $n$ hypersurfaces containing $X$. According to a Bertini type theorem, $d$ of these hypersurfaces  intersect in the scheme $X$ and a residual scheme $R_d$ of codimension $d$, where $d=n-k, \ldots, n$. Furthermore, one obtains the following relations between the degrees of the Segre classes of $X$ and the degrees of the residuals
\begin{equation*}
 \mathrm{deg}\left(s_p(X, \mathbb{P}^n)\right) = m^d - \mathrm{deg}(R_d) - \sum_{i=0}^{p-1} \binom{d}{p-i} m^{p-i} \mathrm{deg}\left(s_i(X, \mathbb{P}^n)\right),
\end{equation*}
for $p = d-(n-k)$. One observes that these relations form an upper-triangular linear equation system with ones on the diagonal. Hence the degrees of the Segre classes can be computed easily once the degrees of the residuals are known.

We continue with describing the regenerative cascade developed by Hauenstein, Sommese and Wampler in \cite{HSW}. It is a method to compute the so-called numerical irreducible decomposition of a variety, i.e., to compute (and represent in a speci\-fied way) the components of the solution set of a polynomial equation system. As we will see, the degrees of the residuals are a by-product of these computations. Let a polynomial equation system be given by homogeneous polynomials $f_1, \ldots, f_n \in \mathbb{C}[x_0, \ldots, x_n]$ of the same degree $m$. The aim of the computations is a set of points containing so-called witness sets for each component of the numerical irreducible decomposition of $V(f_1, \ldots, f_n) \subset \mathbb{P}^n$. A witness set for a component is the intersection of the component with a general linear space of complementary dimension, hence a number of generic points on the component.  In general, the polynomial equation system need not have as many equations as the dimension of the 
ambient space and the equations can have different degrees. However, this special case 
is sufficient for our purposes. Let $L_1, \ldots, L_n$ be random linear functions in $\mathbb{C}[x_0, \ldots, x_n]$ and let $\mathcal{F}_k = [f_1, \ldots, f_k, L_{k+1}, \ldots, L_{n}]^T$ for $k=1, \ldots, n$. 
A point in $V(\mathcal{F}_k)$ is called a non-solution if it does not lie in $V(f_1, \ldots, f_n)$. The non-solutions are nonsingular and isolated, and their number is a by-product of the regenerative cascade computations. Geometrically, all solutions of $\mathcal{F}_k$ are obtained by intersecting the variety $V(f_1, \ldots, f_k)$ with a general linear space of codimension $k$. If the polynomials $f_1, \ldots, f_n$ are chosen in the same way as for the Segre class algorithm above, the non-solutions are the intersection of the residual with the linear space. As the residual is pure-dimensional, the number of non-solutions is the degree of the residual. The by-product of the regenerative cascade, the number of non-solutions of $\mathcal{F}_k$, is hence exactly the information needed to compute the degrees of the Segre classes of $X$.

\begin{ex}
The twisted cubic $C_\mathrm{tw}$ is given by the ideal $(xz-y^2,yw-z^2, xw-yz) \subset \mathbb{C}[x,y,z,w]$. We randomly choose three polynomials $ g_1, g_2, g_3$ of degree $r=2$ in the ideal. By using the implementation of the regenerative cascade in Bertini \cite{B}, we compute that the number of non-solutions of $[g_1, g_2, L_3]^T$ is $\mathrm{deg}(R_2) = 1$, and the number of non-solutions of $[g_1, g_2,g_3]^T$ is $\mathrm{deg}(R_3) = 0$. This is consistent with the fact that two general degree 2 elements of the ideal of the twisted cubic cut out the twisted cubic and a line, and that three general degree 2 elements cut out only the twisted cubic. Hence the degrees of the Segre classes of the twisted cubic can be computed in the following way:
\begin{align*}
 \mathrm{deg}(s_0(C_\mathrm{tw}, \mathbb{P}^3)) &= m^d - \mathrm{deg}(R_2) = 2^2 - 1 = 3, \\
 \mathrm{deg}(s_1(C_\mathrm{tw}, \mathbb{P}^3)) &= m^d - \mathrm{deg}(R_3) - \binom{d}{1} m \cdot \mathrm{deg}(s_0) = 2^3 - 0 - 3\cdot 2 \cdot 3 = -10.
\end{align*}
\end{ex}

\section{Implementation and comparison to other implementations}

As said above, the algorithms described in this article are implemented in the Macaulay2 \cite{M2} package CharacteristicClasses. The implementation is described in a more detailed way in \cite{J}.

We compare the running times of our implementation to Aluffi's implementation CSM, December 2011 version, and the routine euler from Macaulay2. Observe that the latter only works for nonsingular varieties and takes a projective variety as input. The command CSMclass from our implementation uses the algorithm described here. The command CSM from Aluffi's implementation with the same name uses the algorithm described in \cite{A}, and the command euler computes the topological Euler characteristic from the Hodge numbers. We use examples similar to the ones in \cite{EJP}, and a two core processor with 1.40GHz and 4MB RAM. The results are summarized in Table \ref{tab:benchmarks}.

\begin{table}[ht] 
\caption{Comparison of run times. }
\label{tab:benchmarks} 
\begin{tabular}{p{5cm}p{2cm}p{2cm}llllllll}
\hline 
Input & CSMclass symbolic & CSMclass numeric & CSM & euler \\
\hline
twisted cubic & $<1$s & 47s & 2s & $<1$s\\
smooth surface in $\mathbb{P}^4$ defined by minors & 34s & 2131s & 14777s & 88s\\
Segre embedding of $\mathbb{P}^1 \times \mathbb{P}^2$ in $\mathbb{P}^{5}$ & 10s & 285s  & 1s & $<1$s\\
\hline
\end{tabular}
\end{table}

The equations defining the different examples are found on the author's webpage, \texttt{www.math.su.se/$\sim$jost/examplesCSM.m2}. All computations were done over the rational numbers. The ideal of the smooth surface is generated by the 2-by-2 minors of a 2-by-3 matrix of random linear forms.

As Table \ref{tab:benchmarks} shows, the different symbolic implementations complement each other. As the run time of both CSMclass and CSM is exponential in the number of generators, the command euler may be the best choice for examples with many generators. On the other hand, euler only works for smooth varieties. The numerical implementation is actually slower than the symbolic implementation for small examples such as shown in Table \ref{tab:benchmarks}. However, for very large examples the symbolic methods may not terminate due to insufficient memory. The numerical implementation does not have this problem. Furthermore, the numerical algorithm is parallelizable as the main step, the regenerative cascade, is parallelizable.

The complexity of computing the topological Euler characteristic has been studied in \cite{BCL} and been found to be complete in the class $\mathrm{FP}_\mathbb{C}^{\#\mathrm{P}_\mathbb{C}}$. Roughly speaking, functions in $\#\mathrm{P}_\mathbb{C}$ are polynomial time counting functions and functions in $\mathrm{FP}_\mathbb{C}^{\#\mathrm{P}_\mathbb{C}}$ are functions computable in polynomial time using oracle calls to functions in $\#\mathrm{P}_\mathbb{C}$. For the precise definition of the class, we refer to, e.g.,\ \cite{BCL}.

\section{Example from algebraic statistics}

We illustrate the algorithm for the numerical computation of the Euler characteristic by an example from algebraic statistics which uses the topological Euler characteristic: the maximum likelihood degree, defined in \cite{CHKS}. We compute the maximum likelihood degree of the statistical model described in Example 2.2.2 in \cite{DSS} using a theorem from \cite{H} together with our implementation of the algorithm presented here. 


To introduce the maximum likelihood degree properly is beyond the scope of this article. The reader interested in details may consult \cite{DSS} or \cite{PS}. 

In algebraic statistics, a statistical model for discrete data is given by a subvariety $X$ of $\mathbb{P}^n=\mathbb{P}^n_\mathbb{C}$. In this context, we focus on implicit models, i.e., subvarieties given by equations. Denote by $p_0, \ldots, p_n$ the homogeneous coordinates of $\mathbb{P}^n$, where $p_i$ represents the probability of the $i$-th event. Of course, the statistically interesting case is that all coordinates are real and positive. 
A data vector is given by natural numbers $u_0, \ldots, u_n$, where $u_i$ is the number of times the $i$-th event was observed. The problem of maximum likelihood estimation is to find the point in the model which best explains the data vector $u$. This is done by maximizing the likelihood function 
\begin{equation*}
 L(p_0, \ldots, p_n) = p_0^{u_0} \cdot \ldots \cdot p_n^{u_n}.
\end{equation*}
One usually assumes that none of the probabilities vanishes, and that they sum up to 1. Hence one usually maximizes the likelihood function on the open subvariety $U$ given by
\begin{equation*}
 U = \{x \in X \vert p_0 \cdot  \ldots \cdot p_n (p_0+ \ldots + p_n) \neq 0 \}.
\end{equation*}
The likelihood function can in general have several critical points. Hence methods finding only local maxima such as the Newton method and its derivates may run into problems. It thus makes sense to define the \emph{maximum likelihood degree} of a model to be the number of critical points for a general choice of the data vector $u$. The maximum likelihood degree was introduced in \cite{CHKS}. Since then considerable amount of research has been done computing the maximum likelihood degree of certain classes of models.

In \cite{H}, Huh generalizes a conjecture of Varchenko on the maximum likelihood degree of hyperplane arrangements to smooth very affine varieties. A very affine variety is one that can be embedded into an algebraic torus. The following theorem is a special case of Theorem 1 in \cite{H}.
\begin{thm}[Huh]
 If the exponents $u_i$ are sufficiently general, then the number of critical points of the likelihood function is equal to the signed Euler characteristic $(-1)^{\mathrm{dim}(U)} \chi(U)$.
\end{thm}

\begin{ex}
We use this theorem on the random censoring model with two events described in Example 2.2.2 in \cite{DSS}. It is given by the ideal $I=(2p_0p_1p_2 + p_1^2p_2 + p_1p_2^2 - p_0^2p_{12} + p_1p_2p_{12})$ in the polynomial ring $\mathbb{C}[p_0,p_1,p_2,p_{12}]$. With our implementation of the algorithm described here, we compute the Euler characteristic of $V(I)$ to be 5. As we are interested in the Euler characteristic of $U= V(I) \setminus V(p_0p_1p_2p_{12}(p_0+p_1+p_2+p_{12}))$, we compute the Euler characteristic of $V(p_0p_1p_2p_{12}(p_0+p_1+p_2+p_{12})) \cap V(I)$, which is 2. So the Euler characteristic of $U$ is $5-2=3$, and we get that the maximum likelihood degree of the model is 3. This agrees with the result in \cite{DSS}. The computations for this example are presented in a more concrete way in Example 3 in \cite{J}. 
\end{ex}

\section*{Acknowledgements}
I would like to thank my advisor Sandra Di Rocco for many helpful discussions. Also thanks to David Eklund for helpful discussions. That the residuals may be a by-product of the regenerative cascade was pointed out by Jon Hauenstein. Thanks to Greg Smith for pointing out that the equation system in the main theorem may be inverted explicitely. Finally thanks to Per Alexandersson, Jörgen Backelin and Jens Forsgård for help with the combinatorics.

\end{document}